\documentclass[12pt]{amsart}

\usepackage[top=1in, bottom=1in, left=1in, right=1in]{geometry}
\usepackage{times}

\usepackage{amssymb,amsmath,amsthm}
\usepackage{enumitem}

\makeatletter
\renewcommand\subsection{\@startsection{subsection}{2}%
	\z@{1\linespacing\@plus1\linespacing}{\linespacing}%
	{\normalfont\bfseries\centering}}
\makeatother

\setlist[1]{itemsep=0.5em, topsep=0.5em}

\usepackage{color}
\definecolor{red}{rgb}{1,0,0}
\definecolor{orange}{rgb}{0.7,0.3,0}
\definecolor{blue}{rgb}{0,.3,.7}
\definecolor{green}{rgb}{0,.6,.4}
\usepackage[colorlinks=true, linkcolor=blue, citecolor=green, urlcolor=purple]{hyperref}   

\renewcommand{\le}{\leqslant}

\renewcommand{\ge}{\geqslant}

\numberwithin{equation}{section}

\theoremstyle{plain}
\newtheorem{thm}{Theorem}[section]
\newtheorem{cor}[thm]{Corollary}
\newtheorem{lem}[thm]{Lemma}
\newtheorem{prop}[thm]{Proposition}

\newtheorem*{DSC}{The Duffin--Schaeffer conjecture}
\newtheorem*{catlin-conj}{Catlin's conjecture}
\newtheorem*{model problem}{Model Problem}

\theoremstyle{definition}

\newtheorem*{rem*}{Remark}

\newtheorem*{notation}{Notation}

\newcommand{\N}{\mathbb{N}}
\newcommand{\Z}{\mathbb{Z}}
\newcommand{\Q}{\mathbb{Q}}
\newcommand{\R}{\mathbb{R}}

\newcommand{\E}{\mathbb{E}}
\renewcommand{\P}{\mathbb{P}}

\newcommand{\CA}{\mathcal{A}}
\newcommand{\CB}{\mathcal{B}}

\newcommand{\CD}{\mathcal{D}}
\newcommand{\CE}{\mathcal{E}}
\newcommand{\CF}{\mathcal{F}}

\newcommand{\CL}{\mathcal{L}}

\newcommand{\CP}{\mathcal{P}}

\newcommand{\CR}{\mathcal{R}}
\newcommand{\CS}{\mathcal{S}}

\newcommand{\CV}{\mathcal{V}}
\newcommand{\CW}{\mathcal{W}}

\newcommand{\eq}[1]{ \begin{equation}\begin{split}
		#1 \end{split}
\end{equation} }
\newcommand{\al}[1]{\begin{align} #1 \end{align} }
\newcommand{\als}[1]{\begin{align*} #1 \end{align*} }

\newcommand{\nn}{\nonumber \\}
\newcommand{\ds}{\displaystyle}

\newcommand{\supp}{\operatorname{supp}}
\newcommand{\lcm}{\operatorname{lcm}}
\newcommand{\m}{\operatorname{meas}}

\newcommand{\fl}[1]{\left\lfloor#1\right\rfloor}

\newcommand{\eps}{\varepsilon}
\renewcommand{\phi}{\varphi}

\newcommand{\bs}\boldsymbol{}

\begin{document}

\title{Rational approximations of irrational numbers}

\author{Dimitris Koukoulopoulos}
\address{D\'epartement de math\'ematiques et de statistique\\
	Universit\'e de Montr\'eal\\
	CP 6128 succ. Centre-Ville\\
	Montr\'eal, QC H3C 3J7\\
	Canada}
\email{dimitris.koukoulopoulos@umontreal.ca}

\date{\today}

\begin{abstract}
Given quantities $\Delta_1,\Delta_2,\dots\geqslant 0$, a fundamental problem in Diophantine approximation is to understand which irrational numbers $x$ have infinitely many reduced rational approximations $a/q$ such that $|x-a/q|<\Delta_q$. Depending on the choice of $\Delta_q$ and of $x$, this question may be very hard. However, Duffin and Schaeffer conjectured in 1941 that if we assume a ``metric'' point of view, the question is governed by a simple zero--one law: writing $\varphi$ for Euler's totient function, we either have $\sum_{q=1}^\infty \varphi(q)\Delta_q=\infty$ and then almost all irrational numbers (in the Lebesgue sense) are approximable, or $\sum_{q=1}^\infty\varphi(q)\Delta_q<\infty$ and almost no irrationals are approximable. We present the history of the Duffin--Schaeffer conjecture and the main ideas behind the recent work of Koukoulopoulos--Maynard that settled it.
\end{abstract}

\maketitle

\section{Diophantine approximation}\label{sec:DA}

Let $x$ be an irrational number. In many settings, practical and theoretical, it is important to find fractions $a/q$ of small numerator and denominator that approximate it well. This fundamental question lies in the core of the field of \emph{Diophantine approximation}.

\subsection{First principles}\label{sec:dirichlet}

 The ``high-school way'' of approximating $x$ is to use its decimal expansion. This approach produces fractions $a/10^n$ such that $|x-a/10^n|\approx 10^{-n}$ typically. However, the error can be made much smaller if we allow more general denominators \cite[Theorem 2.1]{harman}.

\begin{thm}\label{thm:Dirichlet}
If $x\in\R\setminus\Q$, then $|x-a/q|<q^{-2}$ for infinitely many pairs $(a,q)\in\Z\times\N$.
\end{thm}

Dirichlet (c.~1840) gave a short and clever proof of this theorem. However, his argument is non-constructive because it uses the pigeonhole principle. This gap is filled by the theory of continued fractions (which actually precedes Dirichlet's proof).

Given any $x\in\R\setminus\Q$, we may write $x=n_0+r_0\approx n_0$, where $n_0=\fl{x}$ is the integer part of $x$ and $r_0=\{x\}$ is its fractional part. We then let  $n_1=\fl{1/r_0}$ and $r_1=\{1/r_0\}$, so that $x=n_0+1/(n_1+r_1)\approx n_0+1/n_1$.  If we repeat this process $j-1$ more times, we find that
 \eq{\label{eq:CFexpansion}
	x\approx 
		n_0+\cfrac{1}{n_1+\cfrac{1}{ \cdots+\cfrac{1}{n_j} } } 
		\qquad \text{with}\quad n_i=\Big\lfloor\frac{1}{r_{i-1}} \Big\rfloor,\ r_i=\Big\{\frac{1}{r_{i-1}} \Big\} \quad\text{for}\ i=1,\dots,j.
}
If we write this fraction as $a_j/q_j$ in reduced form, then a calculation reveals that
\eq{\label{eq:CF-recursion}
\begin{array}{lll} 
	a_j = n_j a_{j-1}+a_{j-2} \quad (j\ge2),& a_1=n_0n_1+1, &a_0=n_0;\\
	q_j  = n_j q_{j-1}+q_{j-2} \quad (j\ge2),& q_1=n_1,& q_0=1.
\end{array}
}
When $j\to\infty$, the right-hand side of \eqref{eq:CFexpansion}, often denoted by $[n_0;n_1,\dots,n_j]$, converges to $x$. The resulting representation of $x$ is called its \emph{continued fraction expansion}. The quotients $a_j/q_j$ are called the \emph{convergents} of this expansion and they have remarkable properties \cite{Khinchin-book}. We list some of them below, with the first one giving a constructive proof of Theorem \ref{thm:Dirichlet}.

\begin{thm}\label{thm:CF}
Assume the above set-up and notations.
	\begin{enumerate}
	\item For each $j\ge0$, we have $1/(2q_jq_{j+1})\le|x-a_j/q_j|\le 1/(q_jq_{j+1})$.
	\item For each $j\ge0$, we have $|x-a_j/q_j|=\min\{|x-a/q|: 1\le q\le q_j\}$.
	\item If $|x-a/q|<1/(2q^2)$ with $a$ and $q$ coprime, then $a/q=a_j/q_j$ for some $j\ge0$.
\end{enumerate}

\end{thm}

\subsection{Improving Dirichlet's approximation theorem}

It is natural to ask when a qualitative improvement of Theorem \ref{thm:Dirichlet} exists. Inverting this question leads us to the following definition: we say that a real number $x$ is \emph{badly approximable} if there is $c=c(x)>0$ such that $|x-a/q|\ge cq^{-2}$ for all $(a,q)\in\Z\times\N$.

We can characterize approximable numbers in terms of their continued fraction expansion. Indeed, Theorem \ref{thm:CF}(a) and relation \eqref{eq:CF-recursion} imply that $1/4\le n_{j+1}q_j^2|x-a_j/q_j|\le1$. Hence, together with Theorem \ref{thm:CF}(c), this implies that $x$ is badly approximable if, and only if, the sequence $(n_j)_{j=0}^\infty$ is bounded. Famously, Lagrange  proved that the quadratic irrational numbers are in one--to--one correspondence with the continued fractions that are eventually periodic \cite[\S 10]{Khinchin-book}. In particular, all such numbers are badly approximable.

\smallskip

A related concept to badly approximable numbers is the \emph{irrationality measure}. For each $x\in\R$, we define it to be
\[
\mu(x):=\sup\{\nu\ge0: \mbox{$0<|x-a/q|<q^{-\nu}$ for infinitely many pairs $(a,q)\in\Z\times\N$}\} .
\]
Note that $\mu(x)=1$ if $x\in\Q$, whereas $\mu(x)\ge2$ if $x\in\R\setminus\Q$ by Theorem \ref{thm:Dirichlet}. Moreover, $\mu(x)=2$ if $x$ is badly approximable. In particular, $\mu(x)=2$ for all quadratic irrationals $x$. Remarkably, Roth \cite{roth-algebraic} proved that $\mu(x)=2$ for all algebraic irrational numbers $x$.

Determining the irrationality measure of various famous transcendental constants is often very hard. We do know that $\mu(e)=2$, where $e$ denotes Euler's constant. However, determining $\mu(\pi)$ is a famous open problem. Towards it, Zeilberger and Zudilin \cite{ZZ} proved that $\mu(\pi)\le 7.10320533\dots$. It is widely believed that $\mu(\pi)=2$.

\smallskip

Instead of trying to reduce the error term in Dirichlet's approximation theorem, we often require a different type of improvement: restricting the denominators $q$ to lie in some special set $\CS$. The theory of continued fractions is of limited use for such problems, because the denominators it produces satisfy rigid recursive relations (cf.~\eqref{eq:CF-recursion}).

For rational approximation with prime or square denominators, the best results at the moment are due to Matom\"aki \cite{matomaki} and Zaharescu \cite{zaharescu}, respectively.

\begin{thm}[Matom\"aki (2009)] \label{thm:matomaki}
Let $x$ be an irrational number and let $\eps>0$. There are infinitely many integers $a$ and prime numbers $p$ such that $|x-a/p|<p^{-4/3+\eps}$.
\end{thm}

\begin{thm}[Zaharescu (1995)] \label{thm:zaharescu}
Let $x$ be an irrational number and let $\eps>0$. There are infinitely many pairs $(a,q)\in\Z\times\N$ such that $|x-a/q^2|<q^{-8/3+\eps}$.
\end{thm}

Two important open problems are to show that Theorems \ref{thm:matomaki} and \ref{thm:zaharescu} remain true even if we replace the constants $4/3$ and $8/3$ by $2$ and $3$, respectively.

\section{Metric Diophantine approximation}

Unable to answer simple questions about the rational approximations of specific numbers, a lot of research adopted a more statistical point of view. For example, given $M>2$, what \emph{proportion} of real numbers have irrationality measure $\ge M$? This new perspective gives rise to the theory of \emph{metric Diophantine approximation}, which has a much more analytic and probabilistic flavor than the classical theory of Diophantine approximation. As we will see, the ability to ignore small pathological sets of numbers leads to a much more robust theory that provides simple and satisfactory answers to very general questions.

In order to give precise meaning to the word ``proportion'', we shall endow $\R$ with a measure. Here, we will mainly use the Lebesgue measure (denoted by ``meas'').

\subsection{The theorems of Khinchin and Jarn\'ik-Besicovitch}\label{sec:khinchin}

The foundational result in the field of metric Diophantine approximation was proven by Khinchin in his seminal 1924 paper \cite{Khinchin-paper}. It is a rather general result: given a sequence $\Delta_1,\Delta_2,\dots,\ge0$ of ``permissible margins of error'', we wish to determine for which real numbers $x$ there are infinitely many pairs $(a,q)\in\Z\times\N$ such that $|x-a/q|<\Delta_q$. Clearly, if $x$ has this property, so does $x+1$. Hence, we may focus on studying 
\eq{\label{eq:dfn-A}
\CA:= \big\{x\in[0,1] : \mbox{$|x-a/q|<\Delta_q$ for infinitely many pairs $(a,q)\in\Z\times \N$}\big\}.
}

\begin{thm}[Khinchin (1924)]\label{thm:Khinchin} Let $\Delta_1,\Delta_2,\dots\ge0$ and let $\CA$ be defined as in \eqref{eq:dfn-A}.
	\begin{enumerate}
		\item If $\sum_{q=1}^\infty q\Delta_q<\infty$, then $\m(\CA)=0$.
		\item If $\sum_{q=1}^\infty q\Delta_q=\infty$ and the sequence $(q^2\Delta_q)_{q=1}^\infty$ is decreasing, then $\m(\CA)=1$.
	\end{enumerate}

\end{thm}

\begin{cor}\label{cor:Khinchin} For almost all $x\in\R$, we have $|x-a/q|\le 1/(q^2\log q)$ for infinitely many $(a,q)\in\Z\times\N$. On the other hand, if $c>1$ is fixed, then for almost every $x\in\R$, the inequality $|x-a/q|\le 1/(q^2\log^cq)$ admits only finitely many solutions $(a,q)\in\Z\times \N$.
\end{cor}

In particular, Corollary \ref{cor:Khinchin} implies that the set of badly approximable numbers has null Lebesgue measure. On the other hand, it also says that almost all real numbers have irrationality measure equal to 2. This last result is the main motivation behind the conjecture that $\mu(\pi)=2$: we expect $\pi$ to behave like a ``typical'' real number.

Naturally, the fact that $\CW_M:=\{x\in\R:\mu(x)\ge M\}$ has null Lebesgue measure for $M>2$ raises the question of determining its Hausdorff dimension (denoted by $\dim(\CW_M)$). Jarn\'ik \cite{jarnik} and Besicovitch \cite{besicovitch} answered this question independently of each other.

\begin{thm}[Jarn\'ik (1928), Besicovitch (1934)]\label{thm:jarnik-besicovitch}
	We have $\dim(\CW_M)=2/M$ for all $M\ge2$.
\end{thm}

\subsection{Generalizing Khinchin's theorem}

Following the publication of Khinchin's theorem, research focused on weakening the assumption that $q^2\Delta_q\searrow$ in part (b). Importantly, doing so would open the door to understanding rational approximations using only a restricted set of denominators. Indeed, if $q^2\Delta_q\searrow$, then either $\Delta_q>0$ for all $q$, or there is  $q_0$ such that $\Delta_q=0$ for all $q\ge q_0$. The second case is trivial, since it implies $\CA=\emptyset$. So, if we wish to understand Diophantine approximation with a restricted set of denominators $\CS$ (which would require $\Delta_q=0$ for $q\notin\CS$), then we must prove a version of Theorem \ref{thm:Khinchin}(b) without the assumption that $q^2\Delta_q\searrow$.

In order to understand better the forces at play here, it is useful to recast Khinchin's theorem in probabilistic terms. For each $q$, let us define the set
\al{
	\CA_q &:= \big\{x\in[0,1] : \mbox{there is $a\in\Z$ such that $|x-\frac{a}{q}|<\Delta_q$}\big\} \nn
	&= [0,1]\cap \bigcup_{0\le a\le q} \Big(\frac{a}{q}-\Delta_q,\frac{a}{q}+\Delta_q\Big) .\label{eq:dfn-Aq}
}
Then $\CA=\{x\in[0,1]: \mbox{$x\in \CA_q$ infinitely often}\}$, which we often write as $\CA=\limsup_{q\to\infty}\CA_q$. We may thus view $\CA$ as the event that for a number chosen uniformly at random from $[0,1]$, an infinite number of the events $\CA_1,\CA_2,\dots$ occur. A classical result from probability theory due to Borel and Cantelli \cite[Lemmas 1.2 \& 1.3]{harman} studies precisely this kind of questions.

\begin{thm} \label{thm:Borel--Cantelli}
	Let $(\Omega,\CF,\P)$ be a probability space, let $E_1,E_2,\dots$ be events in that space, and let $E=\limsup_{j\to\infty}E_j$ be the event that infinitely many of the $E_j$'s occur.
	\begin{enumerate}
		\item (The first Borel--Cantelli lemma) If $\sum_{j=1}^\infty \P(E_j)<\infty$, then $\P(E)=0$.
		\item (The second Borel--Cantelli lemma) If $\sum_{j=1}^\infty \P(E_j)=\infty$ and the events $E_1,E_2,\dots$ are mutually independent, then $\P(E)=1$.
	\end{enumerate}
\end{thm}

\begin{rem*}
Let $N$ be the random variable that counts how many of the events $E_1,E_2,\dots$ occur. We have $\E[N]=\sum_{j=1}^\infty \P(E_j)$. Hence, Theorem \ref{thm:Borel--Cantelli} says that, under certain assumptions, $N=\infty$ almost surely if, and only if, $\E[N]=\infty$.
\end{rem*}

To use the above result in the set-up of Khinchin's theorem, we let $\Omega=[0,1]$ and equip it with the Lebesgue measure as its probability measure. The relevant events $E_j$ are the sets $\CA_q$. Notice that if $\Delta_q>1/(2q)$, then $\CA_q=[0,1]$, in which case $\CA_q$ occurs immediately for all $x\in[0,1]$. In order to avoid these trivial events, we will assume from now on that 
\eq{\label{eq:Deltaq<1/2q}
	\Delta_q\le1/(2q)\quad\text{for all}\ q\ge1,
	\quad\text{whence}\quad \m(\CA_q) = 2q\Delta_q.
}
In particular, we see that part (a) of Khinchin's theorem is a direct consequence of the first Borel--Cantelli lemma. On the other hand,  the second Borel--Cantelli lemma relies crucially on the assumption that the events $E_j$ are independent of each other, something that fails generically for the events $\CA_q$. However, there are variations of the second Borel--Cantelli lemma, where the assumption of independence can be replaced by weaker quasi-independence conditions on the relevant events (cf.~Section \ref{sec:Borel--Cantelli}). From this perspective, part (b) of Khinchin's theorem can be seen as saying that the condition that the sequence $(q^2\Delta_q)_{q=1}^\infty$ is decreasing guarantees enough approximate independence between the events $\CA_q$ so that the conclusion of the second Borel--Cantelli lemma remains valid.

\smallskip

In 1941, Duffin and Schaeffer published a seminal paper \cite{DS} that studied precisely what is the right way to generalize Khinchin's theorem so that the simple zero--one law of Borel--Cantelli holds. Their starting point was the simple observation that certain choices of the quantities $\Delta_q$ create many dependencies between the sets $\CA_q$, thus rendering many of the denominators $q$ redundant. Indeed, note for example that if $\Delta_3=\Delta_{15}$, then $\CA_3\subseteq\CA_{15}$ because each fraction with denominator $3$ can also be written as a fraction with denominator $15$. By exploiting this simple idea, Duffin and Schaeffer proved the following result:

\begin{prop} \label{prop:DS-counterexample}  There are $\Delta_1,\Delta_2,\dots\ge0$ such that $\sum_{q=1}^\infty q\Delta_q=\infty$ and yet $\m(\CA)=0$.
\end{prop}

\begin{proof}
Let $p_1<p_2<\cdots$ be the primes in increasing order, let $q_j=p_1\cdots p_j$, and let $\CS_j=\{dp_j: d|q_{j-1}\}$. We then set $\Delta_q=(q_jj\log^2 j)^{-1}$ if $q\in\CS_j$ for some $j\ge2$; otherwise, we set $\Delta_q=0$. We claim that this choice satisfies the needed conditions.

Since $\CA_q\subseteq\CA_{q_j}$ for all $q\in\CS_j$, we have $\CA=\limsup_{j\to\infty}\CA_{q_j}$. In addition, since $\sum_{j=1}^\infty q_j \Delta_{q_j}<\infty$,  we have $\m(\limsup_{j\to\infty}\CA_{q_j})=0$ by Theorem \ref{thm:Khinchin}(a). Hence, $\m(\CA)=0$, as needed.  On the other hand, we have that
\als{
\sum_{q\ge1}q\Delta_q = \sum_{j\ge2} \sum_{d|q_{j-1}} dp_j\cdot \frac{1}{q_jj\log^2j} 
	=  \sum_{j\ge2} \frac{1}{j\log^2j} \prod_{i\le j-1} \Big(1+\frac{1}{p_i}\Big) ,
}
By the Prime Number Theorem \cite[Theorem 8.1]{dk-book}, the last product is $\ge c\log j$ for some absolute constant $c>0$. Consequently, $\sum_{q=1}^\infty q\Delta_q=\infty$, as claimed.
\end{proof}

In order to avoid the above kind of counterexamples to the generalized Khinchin theorem, Duffin and Schaeffer were naturally led to consider a modified set-up, where only reduced fractions are used as approximations. They thus defined
\eq{\label{eq:dfn-A*}
	\CA^*:= \big\{x\in[0,1] : \mbox{$|x-a/q|<\Delta_q$ for infinitely many reduced fractions $a/q$}\big\}.
}
We may write $\CA^*$ as the lim sup of the sets
\eq{\label{eq:dfn-Aq*}
	\CA_q^*	:= [0,1]\cap \bigcup_{\substack{0\le a\le q \\ \gcd(a,q)=1}} 
	\Big(\frac{a}{q}-\Delta_q,\frac{a}{q}+\Delta_q\Big) .
}
Assuming that \eqref{eq:Deltaq<1/2q} holds, we readily find that
\[
\m(\CA_q^*) = 2\phi(q)\Delta_q
\]
where
\[
\phi(q):=\#\{1\le a\le q: \gcd(a,q)=1\}
\]
is Euler's totient function. They then conjectured that the sets $\CA_q^*$ have enough mutual quasi-independence so that a simple zero--one law holds, as per the Borel--Cantelli lemmas.

\begin{DSC}
	Let $\Delta_1,\Delta_2,\dots\ge0$ and let $\CA^*$ be defined as in \eqref{eq:dfn-A*}.
	\begin{enumerate}
		\item If $\sum_{q=1}^\infty \phi(q)\Delta_q<\infty$, then $\m(\CA^*)=0$.
		\item If $\sum_{q=1}^\infty \phi(q)\Delta_q=\infty$, then $\m(\CA^*)=1$.
	\end{enumerate}
\end{DSC}

Of course, part (a) follows from Theorem \ref{thm:Borel--Cantelli}(a); the main difficulty is to prove (b).

\smallskip

The Duffin--Schaeffer conjecture is strikingly simple and general. Nonetheless, it does not answer our original question: what is the correct generalization of Khinchin's theorem, where we may use non-reduced fractions? This gap was filled by Catlin \cite{catlin}.

\begin{catlin-conj}
	Let $\Delta_1,\Delta_2,\dots\ge0$, let $\Delta_q'=\sup_{m\ge1}\Delta_{qm}$, and let $\CA$ be as in \eqref{eq:dfn-A}.
		\begin{enumerate}
		\item If $\sum_{q=1}^\infty \phi(q)\Delta_q'<\infty$, then $\m(\CA)=0$.
		\item If $\sum_{q=1}^\infty \phi(q)\Delta_q'=\infty$, then $\m(\CA)=1$.
	\end{enumerate}
\end{catlin-conj}

As Catlin noticed, his conjecture is a direct corollary of the one by Duffin and Schaeffer. Indeed, let us consider the set
\[
\CA'= \big\{x\in[0,1] : \mbox{$|x-a/q|<\Delta_q'$ for infinitely many reduced fractions $a/q$}\big\}.
\]
This is the set $\CA^*$ with the quantities $\Delta_q$ replaced by $\Delta_q'$, so we may apply the Duffin--Schaeffer conjecture to it. In addition, it is straightforward to check that
\eq{\label{eq:catlin}
	\CA\setminus\Q=\CA'\setminus\Q 
}
when $\Delta_q\to0$. This settles Catlin's conjecture in this case. On the other hand, if $\Delta_q\not\to0$,  then $\CA=[0,1]$ and $\sum_{q=1}^\infty \phi(q)\Delta_q'=\infty$, so that Catlin's conjecture is trivially true.

\smallskip

Just like in Theorem \ref{thm:jarnik-besicovitch} of Jarn\'ik and Besicovitch, it would be important to also have information about the Hausdorff dimension of the sets $\CA$ and $\CA^*$ in the case when they have null Lebesgue measure. In light of relation \eqref{eq:catlin}, it suffices to answer this question for the latter set. Beresnevich and Velani \cite{hausdorff DS} proved the remarkable result that the Duffin--Schaeffer conjecture implies a Hausdorff measure version of itself. This is a consequence of a much more general \emph{Mass Transference Principle} that they established, and which allows transfering information concerning the Lebesgue measure of certain $\limsup$ sets to the Hausdorff measure of rescaled versions of them. As a corollary, they proved:

\begin{thm}[Beresnevich--Velani (2006)]  Assume that the Duffin--Schaeffer conjecture is true. Let $\Delta_1,\Delta_2,\dots\ge0$ be such that $\sum_{q=1}^\infty \phi(q)\Delta_q<\infty$. Then the Hausdorff dimension of the set $\CA^*$ defined by \eqref{eq:dfn-A*} equals the infimum of the set of $s>0$ such that  $\sum_{q=1}^\infty \phi(q)\Delta_q^s<\infty$.
\end{thm}

\subsection{Progress towards the Duffin--Schaeffer conjecture}

Since its introduction in 1941, the Duffin--Schaeffer conjectured has been the subject of intensive research activity, with various special cases proven over the years. This process came to a conclusion recently with the proof of the full conjecture  \cite{DS-proof}.

\begin{thm}[Koukoulopoulos--Maynard (2020)] \label{thm:DS proof}
	The Duffin--Schaeffer conjecture is true.
\end{thm}

We will outline the main ideas of the proof of Theorem \ref{thm:DS proof} in \S \ref{sec:proof}. But first we give an account of the work that preceeded it.

\begin{notation}
	Given two functions $f,g:X\to\R$, we write $f(x)\ll g(x)$ (or $f(x)=O(g(x))$) for all $x\in X$ to mean that there is a constant $C$ such that $|f(x)|\le Cg(x)$ for all $x\in Y$.
\end{notation}

In the same paper where they introduced their conjecture, Duffin and Schaeffer proved the first general case of it:

\begin{thm}[Duffin--Schaeffer (1941)] \label{thm:DS for normal integers} 	
The Duffin--Schaeffer conjecture is true for all sequences $(\Delta_q)_{q=1}^\infty$ such that
	\eq{\label{eq:DS-condition}
	\limsup_{Q\to\infty} \frac{\sum_{q\le Q} \phi(q)\Delta_q}{\sum_{q\le Q}q\Delta_q} >0 .
	}
\end{thm}

To appreciate this result, we must make a few comments about condition \eqref{eq:DS-condition}. Note that its left-hand side is the average value of $\phi(q)/q$ over $q\in[1,Q]$, where $q$ is weighted by $w_q:=q\Delta_q$. In particular, we may restrict our attention to $q$ with $\Delta_q>0$. Now, we know
\[
\frac{\phi(q)}{q} = \prod_{p|q}\Big(1-\frac{1}{p}\Big) .
\]
In particular, $\phi(q)/q\le1$, and the only way this ratio can become much smaller than 1 is if $q$ is divisible by lots of small primes.  To see this, let us begin by observing that $q$ can have at most $\log q/\log2$ prime factors in total. Therefore,
\eq{\label{eq:big primes}
\prod_{p|q,\ p>\log q} \Big(1-\frac{1}{p}\Big) \ge \Big(1-\frac{1}{\log q}\Big)^{\log q/\log 2} \ge\frac{1}{5}
}
for $q$ large enough. In addition, we have
\eq{\label{eq:big primes2}
	\prod_{\substack{(\log q)^{0.01}<p\le \log q \\  p|q }} \Big(1-\frac{1}{p}\Big) \ge
		\prod_{(\log q)^{0.01}<p\le \log q } \Big(1-\frac{1}{p}\Big) \ge \frac{1}{200}
}
for $q$ large enough by Mertens' estimate \cite[Theorem 3.4]{dk-book}. Already the above inequalities show that only the primes $\le(\log q)^{0.01}$ can affect the size of $\phi(q)/q$. But more is true: $\phi(q)/q$ is small only if $q$ is divided by \emph{many} primes $\le (\log q)^{0.01}$.  Imagine for example that 
\eq{\label{eq:not-too-many-small-primes}
	\#\{p|q: e^{j-1}<p\le e^j\} \le e^{j}/j^2 + 1000
}
for $j=1,2,\dots, 1+\fl{0.01\log\log q}$. We would then have
\[
\prod_{\substack{e^{j-1}<p\le e^j\\ p|q}} \Big(1-\frac{1}{p}\Big) \ge \Big(1-\frac{1}{e^j}\Big)^{e^{j}/j^2+1000} = \exp\big(- 1/j^2+O(e^{-j})\big) .
\]
Multiplying this over all $j$, we deduce that $\phi(q)/q\ge c$ for some $c>0$ independent of $q$. 

We have thus proven that for \eqref{eq:DS-condition} to fail, the main contribution to the weighted sum $\sum_{q\le Q}w_q$ with $w_q=q\Delta_q$ must come from integers for which \eqref{eq:not-too-many-small-primes} fails. As a matter of fact, \eqref{eq:not-too-many-small-primes} must fail for lots of $j$'s. This is an extremely rare event if we choose $q$ \emph{uniformly at random} from $[1,Q]$ (or even if we choose it uniformly at random from various ``nice'' subsets of $[1,Q]$, such as the primes, or the values of a monic polynomial with integer coefficients).  A simple way to see this is to calculate the average value of the function $\#\{p|q: e^{j-1}<p\le e^j\}$ with respect to the uniform counting measure on $[1,Q]$. We have
\[
 	\frac{1}{Q}\sum_{q\le Q} \#\{p|q: e^{j-1}<p\le e^j\} 
 		 = \sum_{e^{j-1}<p\le e^j} \frac{\#\{q\le Q: p|q\} }{Q} 
 		\le \sum_{e^{j-1}<p\le e^j} \frac{1}{p} \ll \frac{1}{j} 
\]
by Mertens' theorem \cite[Theorem 3.4]{dk-book}. This is \emph{much} smaller than $e^j/j^2$, so \eqref{eq:not-too-many-small-primes} should fail rarely as $j\to\infty$. (For instance, we may use Markov's inequality to see this claim.)

In conclusion, Theorem \ref{thm:DS for normal integers} settles the Duffin--Schaeffer conjecture when $\Delta_q$ is mainly supported on ``normal'' integers, without too many small prime factors. In particular, it implies a significant improvement of Theorems \ref{thm:matomaki} and \ref{thm:zaharescu} for almost all $x\in\R$.

\begin{cor} \label{cor:DS for primes}
	For almost all $x\in\R$, there are infinitely many reduced fractions $a/p$ and $b/q^2$ such that $p$ is prime, $|x-a/p|<p^{-2}$ and  $|x-b/q^2|<q^{-3}$.
\end{cor}

The next important step towards the Duffin--Schaeffer conjecture is a remarkable \emph{zero--one law} due to Gallagher \cite{gallagher}.

\begin{thm}[Gallagher (1961)] \label{thm:gallagher} If $\CA^*$ is as in \eqref{eq:dfn-A*}, then $\m(\CA^*)\in\{0,1\}$.
\end{thm}

Gallagher's theorem says \emph{grosso modo} that either we chose the quantities $\Delta_q$ to be ``too small'' and thus missed almost all real numbers, or we chose them ``sufficiently large'' so that almost all numbers have the desired rational approximations. The Duffin--Schaeffer conjecture is then the simplest possible criterion to decide in which case we are.

The proof of Theorem \ref{thm:gallagher} is a clever adaptation of an ergodic-theoretic argument due to Cassels \cite{cassels 0-1} in the simpler setting of non-reduced rational approximations. We give Cassel's proof and refer the interested readers to \cite{gallagher,harman} for the proof of Theorem \ref{thm:gallagher}.

\begin{thm}[Cassels (1950)] \label{thm:cassels} 
If $\CA$ is as in \eqref{eq:dfn-A}, then $\m(\CA)\in\{0,1\}$.
\end{thm}

\begin{proof}
We need the following fact \cite[Lemma 2.1]{harman} that uses Lebesgue's Density Theorem: Let $I_1,I_2,\dots,\ J_1,J_2,\dots$ be intervals of lengths tending to 0, and let $c>0$. For all $k$, suppose $J_k\subseteq I_k$ and  $\m(J_k)\ge c\m(I_k)$. Then $\m(\limsup_{k\to\infty}I_k\setminus \limsup_{k\to\infty}J_k)=0$.

Now, for each $r\ge1$, let $\CA^{(r)}$ be defined as in \eqref{eq:dfn-A} but with $\Delta_q/r$ in place of $\Delta_q$. Hence, $\m(\CA\setminus\CA^{(r)})=0$ by the above fact. Therefore, if $\CA^{(\infty)}:=\bigcap_{n=1}^\infty\CA^{(n)}$, then $\m(\CA\setminus\CA^{(\infty)})=0$. Now, consider the map $\psi:[0,1]\to[0,1]$ defined by $\psi(x):=\{2x\}$, and note that $\psi(\CA^{(\infty)})\subseteq \CA^{(\infty)}$. In particular, $\frac{1}{N}\sum_{n=0}^{N-1}1_{\CA^{(\infty)}}(\psi^n(x))=1$ for all $x\in\CA^{(\infty)}$ and all $N\in\N$. Since $\psi$ is ergodic with respect to the Lebesgue measure \cite[p. 293 \& 305-6]{SS}, Birkhoff's Ergodic Theorem \cite[Ch.~6, Cor.~5.6]{SS} implies that $\m(\CA^{(\infty)})\in\{0,1\}$.
\end{proof}

The first significant step towards establishing the Duffin--Schaeffer conjecture for irregular sequences $\Delta_q$, potentially supported on integers with lots of small prime factors, was carried out by Erd\H os \cite{erdos} and Vaaler \cite{vaaler}.

\begin{thm}[Erd\H os (1970) -- Vaaler (1978)] \label{thm:Erdos-Vaaler} The Duffin--Schaeffer conjecture is true for all sequences $(\Delta_q)_{q=1}^\infty$ such that $\Delta_q=O(1/q^2)$ for all $q$.
\end{thm}

This theorem is of course most interesting when $\sum_{q=1}^\infty \phi(q)\Delta_q=\infty$. Since $\Delta_q=O(1/q^2)$ and $\phi(q)\le q$, we find $\sum_{q\in\CS}1/q=\infty$ with $\CS=\{q:\Delta_q>0\}$. In particular, $\CS$ must be somewhat dense in $\N$. Therefore, Theorem \ref{thm:Erdos-Vaaler} has the advantage over Theorem \ref{thm:DS for normal integers} that $\CS$ can contain many irregular integers, and the disadvantage that it has to be quite dense.

\smallskip

The Duffin--Schaeffer conjecture has a natural analogue in $\R^k$ with $k\ge2$: given $\Delta_1,\Delta_2,\dots\ge0$, let $\CA^*(k)$ be the set of $\vec{x}=(x_1,\dots,x_k)\in\R^k$ for which there are infinitely many $k$-tuples $(a_1/q,\dots,a_k/q)$ of reduced fractions with  $|x_j-a_j/q|<\Delta_q$ for all $j$. Then $\CA^*(k)$ should contain almost no or almost all $\vec{x}\in\R^k$, according to whether the series $\sum_{q=1}^\infty (\phi(q)\Delta_q)^k$ converges or diverges. This was proven by Pollington and Vaughan \cite{PV}.

\begin{thm}[Pollington--Vaughan (1990)] The $k$-dimensional Duffin--Schaeffer conjecture is true for all $k\ge2$. 
\end{thm}

Following this result, a lot of research focused on proving the Duffin--Schaeffer conjecture when the series $\sum_{q=1}^\infty\phi(q)\Delta_q$ diverges fast enough (see, e.g., \cite[Theorem 3.7(iii)]{harman}, \cite{extra-div1, extra-div2}). Aistleitner, Lachmann, Munsch, Technau and Zafeiropoulos \cite{extra-div3} proved the Duffin--Schaeffer conjecture when $\sum_{q=1}^\infty \phi(q)\Delta_q/(\log q)^\eps=\infty$ for some $\eps>0$. A report by Aistleitner \cite{extra-div4}, announced at the same time as \cite{DS-proof}, explains how to replace $(\log q)^\eps$ by $(\log\log q)^\eps$.

\section{The main ingredients of the proof of the Duffin--Schaeffer conjecture}\label{sec:proof}

\subsection{Borel--Cantelli without independence}\label{sec:Borel--Cantelli}

Recall the definition of the sets $\CA_q^*$ in \eqref{eq:dfn-Aq*}. Let us assume that $\Delta_q\le1/(2q)$ for all $q$ (cf.~\eqref{eq:Deltaq<1/2q}) so that $\m(\CA_q^*)=2\phi(q)\Delta_q\in[0,1]$, and let us also suppose that $\sum_{q=1}^\infty \phi(q)\Delta_q=\infty$. The first technical difficulty we must deal with is how to prove an analogue of the second Borel--Cantelli lemma (cf.~Theorem \ref{thm:Borel--Cantelli}(b)) without assuming that the events $\CA_q^*$ are independent. We follow an idea due to Turan, which is already present in \cite{DS}. 

By Gallagher's zero--one law, it is enough to show that $\m(\CA^*)>0$. Since $\bigcup_{q\ge Q} \CA_q^*\searrow \CA^*$, we may equivalently prove that there is some constant $c>0$ such that $\m(\bigcup_{q\ge Q} \CA_q^*)\ge c$ for all large $Q$. In order to limit the potential overlap among the sets $\CA_q^*$, we only consider an appropriate subset of them. Since $\m(\CA_q^*)=2\phi(q)\Delta_q\in[0,1]$ for all $q$, and since $\sum_{q\ge Q}\phi(q)\Delta_q=\infty$, there must exist some $R\ge Q$ such that 
\eq{\label{eq:BC-linear}
1\le \sum_{q\in[Q,R]} \m(\CA_q^*)\le 2.
}
We will only use the events $\CA_q^*$ with $q\in[Q,R]$. We trivially have the union bound
\[
\m\Big(\bigcup_{q\in[Q,R]}\CA_q^*\Big) \le \sum_{q\in[Q,R]}\m(\CA_q^*) \le 2 .
\]
If we can show that the sets $\CA_q^*$ with $q\in[Q,R]$ do not overlap too much, so that
\eq{\label{eq:BC-limited overlap}
\m\Big(\bigcup_{q\in[Q,R]}\CA_q^*\Big) \ge c \sum_{q\in[Q,R]}\m(\CA_q^*) \ge c,
}
we will be able to deduce that $\m(\bigcup_{q\in[Q,R]}\CA_q^*)\ge c$ and \emph{a fortiori} that $\m(\bigcup_{q\ge Q}\CA_q^*)\ge c$. 
As the following lemma shows, \eqref{eq:BC-limited overlap} is true under \eqref{eq:BC-linear} as long as we can control the correlations of the events $\CA_q^*$ \emph{on average}. 

\begin{lem}\label{lem:CS}
Let $E_1,\dots,E_k$ be events in the probability space $(\Omega,\CF,\P)$. We then have that
\[
\P\Big(\bigcup_{j=1}^k E_j\Big) \ge \frac{(\sum_{j=1}^k \P(E_j))^2}{\sum_{i,j=1}^k \P(E_i\cap E_j)} .
\]
\end{lem}

\begin{proof} Let $N=\sum_j 1_{E_j}$. We then have $\E[N]= \sum_j \P(E_j)$. 
	On the other hand, the Cauchy--Schwarz inequality implies that
	\[
	\E[N]^2 = \E [ 1_{N>0} \cdot N]^2 \le \P(\supp(N)) \cdot \E[N^2] .
	\]
	Since $\supp(N)=\bigcup_j E_j$ and $N^2= \sum_{i,j} 1_{E_i\cap E_j}$, the lemma follows. 
\end{proof}

The following proposition summarizes the discussion of this section. 

\begin{prop}\label{prop:DS via CS}
	Let $\Delta_1,\Delta_2,\dots\ge0$, and let $\CA_q^*$ be as in \eqref{eq:dfn-Aq*}.
	\begin{enumerate}
		\item 	If $C>0$ and $R\ge Q\ge1$ are such that
		\eq{\label{eq:DS via correlations}
		1\le \sum_{q\in [Q,R]} \m(\CA_q^*) \le 2
		\qquad\text{and}\qquad
		\sum_{Q\le q<r\le R}\m(\CA_q^*\cap\CA_r^*)\le C,
		}
		then $\m(\bigcup_{q\in[Q,R]}\CA_q^*)\ge 1/(2+2C)$.
		\item If there are infinitely many disjoint intervals $[Q,R]$ satisfying \eqref{eq:DS via correlations} with the same constant $C>0$, 
		then $\m(\limsup_{q\to\infty}\CA_q^*)=1$.
		\end{enumerate}
\end{prop}

%

%
%

\subsection{A bound on the pairwise correlations}

As per Proposition \ref{prop:DS via CS}, we need to control the correlations of the events $\CA_q^*$. To this end, we have a lemma of Pollington--Vaughan \cite{PV} (see also \cite{erdos,vaaler}).

\begin{lem} \label{lem:PV} 
	Let $q,r$ be two distinct integers $\ge2$, let $\Delta_q,\Delta_r\ge0$, let $\CA_q^*, \CA_r^*$ be as in \eqref{eq:dfn-Aq*}, and let  $M(q,r)=2\max\{\Delta_q,\Delta_r\}\lcm[q,r]$. If $M(q,r)\le 1$, then $\CA_q^*\cap\CA_r^*=\emptyset$. Otherwise,
	\[
	\m(\CA_q^*\cap\CA_r^*)\ll \phi(q)\Delta_q \cdot \phi(r)\Delta_r \cdot  \exp\Big(\sum_{\substack{ p|qr/\gcd(q,r)^2 \\  p>M(q,r) }} \frac{1}{p}\Big) .
	\]
\end{lem}

\begin{proof} 
Let $\Delta=\max\{\Delta_q,\Delta_r\}$, $\delta=\min\{\Delta_q,\Delta_r\}$ and $M=M(q,r)$. 
The intervals $I_a=(\frac{a}{q}-\Delta_q,\frac{a}{q}+\Delta_q)$ and $J_b=(\frac{b}{r}-\Delta_r,\frac{b}{r}+\Delta_r)$ intersect only if $2\Delta>|\frac{a}{q}-\frac{b}{r}|$. Since the right-hand side is $\ge 1/\lcm[q,r]$ when $\gcd(a,q)=\gcd(b,r)=1$, we infer that $\CA_q^*\cap\CA_r^*=\emptyset$ if $M\le 1$. 

Now, assume that $M>1$. Since $\m(I_a\cap J_b)\le2\delta$ for all $a,b$, we have
\[
	\m(\CA_q^*\cap \CA_r^*) \le 2\delta \cdot  \#\bigg\{ \begin{array}{ll}1\le a\le q, &\gcd(a,q)=1 \\ 1\le b\le r, &\gcd(b,r)=1 \end{array} 
	:\ \Big|\frac{a}{q}-\frac{b}{r}\Big|<2\Delta\bigg\} .
\]
Let $a/q-b/r=m/\lcm[q,r]$. Then $1\le|m|\le M$ and $\gcd(m,q_1r_1)=1$, where $q_1=q/d$ and $r_1=r/d$ with $d=\gcd(q,r)$. For each such $m$, a straightforward application of the Chinese Remainder Theorem gives that the number of admissible pairs $(a,b)$ is 
%
\[
\le d \prod_{p|\gcd(d,q_1r_1m)}\Big(1-\frac{1}{p}\Big)
\prod_{p|d,\ p\nmid q_1r_1m}\Big(1-\frac{2}{p}\Big)  
\le d \frac{\prod_{p|d}(1-\frac{1}{p})^2}{\prod_{p|\gcd(d,q_1r_1)}(1-\frac{1}{p})} \cdot \frac{|m|}{\phi(|m|)} ,
\]
where we used that $1-2/p\le(1-1/p)^2$. We then sum this inequality over $m$ and use Lemma \ref{lem:sieve upper bound} below to complete the proof. (For full details, see \cite{PV} or \cite[Lemma~2.8]{harman}.)
%
%
\end{proof}

\begin{lem}\label{lem:sieve upper bound} Fix $C\ge1$, and let $(a_p)_{p\ \text{prime}}$ be a sequence taking values in $[0,C]$. Then
		\[
		\sum_{n\le x} \prod_{p|n} a_p  \ll_C x \exp\Big(\sum_{p\le x} \frac{a_p-1}{p} \Big) \qquad\text{for all}\ x\ge1.
		\]
\end{lem}

\begin{proof} See Theorem 14.2 in \cite{dk-book}.
\end{proof}

\subsection{Generalizing the Erd\H os--Vaaler argument}\label{sec:EV}

The next step is to study averages of $\exp(\sum_{p|qr/\gcd(q,r)^2,\ p>M(q,r)}1/p)$. This gets a bit too technical in general, so we focus on the following special case:

\begin{thm}\label{thm:DS-special-case} Let $Q\ge N\ge2$, and let $\CS\subseteq\{Q\le q\le 2Q:q\ \text{square-free}\}$ be such that 
	\eq{\label{eq:DS-total weight}
	N/2\le \sum_{q\in\CS} \frac{\phi(q)}{q} \le N.
	}
	We then have
	\eq{\label{eq:DS bilinear}
	 \sum_{q,r\in\CS } \frac{\phi(q)\phi(r)}{qr} \exp\Big(\sum_{\substack{p|qr/\gcd(q,r)^2 \\ p>Q/[N\gcd(q,r)]}}  \frac{1}{p} \Big) \ll N^2.
	}
	In particular, if $\CA_q^*$ is as in \eqref{eq:dfn-Aq*} with $\Delta_q=1/(qN)$, then $\m(\bigcup_{q\in\CS}\CA_q^*)\gg1$.
\end{thm}

\begin{rem*} To see the last assertion, recall the notation $M(q,r)=2\max\{\Delta_q,\Delta_r\}\lcm[q,r]$ from Lemma \ref{lem:PV}. By the assumptions of the theorem, we have $M(q,r)\ge 2Q/[N\gcd(q,r)]$ for $q,r\in\CS$. Hence, if \eqref{eq:DS bilinear} holds, then $\sum_{q,r\in\CS}\m(\CA_q^*\cap\CA_r^*)\ll 1$ by Lemma \ref{lem:PV}. We may then apply Proposition \ref{prop:DS via CS} to deduce that $\m(\bigcup_{q\in\CS}\CA_q^*)\gg1$.
\end{rem*}

When $N\gg Q$, Theorem \ref{thm:DS-special-case} follows from the work of  Erd\H os and Vaaler (Theorem \ref{thm:Erdos-Vaaler}), but when $N=o(Q)$ it was not known prior to \cite{DS-proof} in this generality. The proof begins by adapting the Erd\H os--Vaaler argument to this more general set-up.

First, we must control the sum over primes in \eqref{eq:DS bilinear}. Using \eqref{eq:not-too-many-small-primes} turns out to be too crude, so we modify our approach. Let $t_j=\exp(2^j)$ and $j_0$ be such that $\sum_{t<p\le t^2}1/p\le 1$ for $t\ge t_{j_0}$ ($j_0$ exists by Mertens' theorems \cite[Theorem 3.4]{dk-book}.) Moreover, let 
\[
\CL(q,r)=\sum_{\substack{p|qr/\gcd(q,r)^2 \\ p>Q/[N\gcd(q,r)]}}\frac{1}{p} ,
\quad \lambda_t(q) = \sum_{\substack{p|q \\ p>t}} \frac{1}{p},\quad
L_t(q,r)= \sum_{\substack{p|qr/\gcd(q,r)^2 \\ p>t}} \frac{1}{p} .
\]
If $L_{t_{j_0}}(q,r)\le 101$, then obviously $\CL(q,r)\ll1$. Otherwise, there is an integer $j\ge j_0$ such that $L_{t_j}(q,r)>101\ge L_{t_{j+1}}(q,r)$. Since $j\ge j_0$, we then also have $L_{t_{j+1}}(q,r)>100$. Now, note that if $Q/[N\gcd(q,r)]\ge t_{j+1}$, then $\CL(q,r)\le L_{t_{j+1}}(q,r)\le 101$.

To sum up, $\CL(q,r)\ll1$, unless $(q,r)\in\CB_{t_{j+1}}$ for some $j\ge j_0$, where
\[
\CB_t := \big\{ (q,r)\in\CS\times \CS : \gcd(q,r)>Q/(Nt) ,\  L_t(q,r)>100\big\} .
\]
We study the contribution of such pairs to the left-hand side of \eqref{eq:DS bilinear}: if $(q,r)\in\CB_{t_{j+1}}$, then
\[
\CL(q,r)\le 101+  \sum_{p\le t_{j+1}} \frac{1}{p} \le \log\log t_{j+1}+O(1) = j\log2 +O(1)
\]
by Mertens' estimate. In conclusion, Theorem \ref{thm:DS-special-case} will follow if we can show that
\eq{\label{eq:DS bilinear 2}
	\sum_{(q,r)\in \CB_t} \frac{\phi(q)\phi(r)}{qr} \ll \frac{N^2}{t}  \qquad\text{for all}\ t\ge t_{j_0+1}.
}

Now, let us consider the special case when $N\gg Q$, which corresponds to the Erd\H os--Vaaler theorem. The inequality $\gcd(q,r)> Q/(Nt)$ is then basically trivially, so we must prove \eqref{eq:DS bilinear 2} by exploiting the condition $L_t(q,r)>100$. Indeed, writing $d=\gcd(q,r)$, $q=dq_1$ and $r=dr_1$, we find that $\lambda_t(q_1)>50$ or $\lambda_t(r_1)>50$. By symmetry, we have
\[
\#\CB_t \le 2\sum_{d\le 2Q} \#\{r_1\le 2Q/d\} \cdot \#\{q_1 \le 2Q/d:  \lambda_t(q_1)>50\}  .
\]
The number of $r_1$'s is of course $\le 2Q/d$. Moreover, using Chernoff's inequality and Lemma \ref{lem:sieve upper bound} with $a_p=\exp(1_{p>t}\cdot t/p)$, we find that
\[
\#\{q_1 \le 2Q/d:  \lambda_t(q_1)>50\} \le \sum_{q_1\le 2Q/d} \exp\big(-50t + t\lambda_t(p)\big) \ll e^{-50t}  Q/d .
\]
Putting everything together, we conclude that 
\eq{\label{eq:DS for full density}
\#\CB_t \ll  e^{-t} Q^2  \qquad\text{for all}\ t\ge 1.
}
In particular, \eqref{eq:DS bilinear 2} holds, thus proving Theorem \ref{thm:DS-special-case} when $N\gg Q$.

On the other hand, if $N=o(Q)$, the condition that $\gcd(q,r)> Q/(Nt)$ for all $(q,r)\in\CB_t$ is non-trivial and we must understand it and exploit it to prove Theorem \ref{thm:DS-special-case}. Indeed, if we treat the weights $\phi(q)/q$ as roughly constant in \eqref{eq:DS-total weight}, we see that $\CS$ contains about $N$ integers from $[Q,2Q]$, i.e., it is a rather sparse set. On the other hand, if $t$ is not too large, then \eqref{eq:DS for full density} gives no savings compared to the trivial upper bound $\#\CB_t\le \#\CS^2\approx  N^2$.

Since the condition that $L_t(q,r)>100$ is insufficient, let us ignore it temporarily and focus on the condition that $\gcd(q,r)> Q/(Nt)$ for all $(q,r)\in\CB_t$. There is an obvious way in which this condition can be satisfied for many pairs $(q,r)\in\CS\times \CS$: if there is some \emph{fixed} integer $d>Q/(Nt)$ that divides a large proportion of integers in $\CS$. Notice that the number of total multiples of $d$ in $[Q,2Q]$ is about $Q/d<t\cdot N$, which is compatible with \eqref{eq:DS-total weight}. We thus arrive at the following key question:

\begin{model problem}
Let $D\ge1$ and $\delta\in(0,1]$, and let $\CS\subseteq[Q,2Q]\cap \Z$ be a set of $\gg \delta Q/D$ elements such that there are $\ge \delta \#\CS^2$ pairs $(q,r)\in\CS\times\CS$ with $\gcd(q,r)>D$. Must there be an integer $d>D$ that divides $\gg \delta^{100}Q/D$ elements of $\CS$?
\end{model problem}

It turns out that the answer to the Model Problem as stated is no. However, a technical variant of it is true, that takes into account the weights $\phi(q)/q$ in \eqref{eq:DS-total weight} and \eqref{eq:DS bilinear 2}, and that is asymmetric in $q$ and $r$. We shall explain this in the next section.

For now, let us assume that the Model Problem as stated has an affirmative answer, and let us see how this yields Theorem \ref{thm:DS-special-case}. Suppose \eqref{eq:DS bilinear 2} fails for some $t$. By the Model Problem, there must exist an integer $d>Q/(Nt)$ dividing $\gg t^{-100}\#\CS$ members of $\CS$. We might then also expect that $\#\CB_t\gg t^{-200} \#\{(dm,dn)\in \CB_t:m,n\ge1\}$. But note that if $(q,r)=(dm,dn)\in\CB_t$, then $m,n\le 2Q/d< 2tN$ and $qr/\gcd(q,r)^2=mn/\gcd(m,n)^2$. In particular, $L_t(m,n)>100$, so the argument leading to \eqref{eq:DS for full density} implies that the number of $(dm,dn)\in\CB_t$ is $\ll e^{-t} t^2 N^2$. Hence, $\CB_t\ll  e^{-t}t^{202} N^2 \ll N^2/t$, as needed.

\subsection{An iterative compression algorithm}

To attack the Model Problem, we view it as a question in graph theory: consider the graph $G$, with vertex set $\CS$ and edge set $\CB_t$. If the edge density of $G$ is $\ge 1/t$, must there exist a dense subgraph $G'$ all of whose vertices are divisible by an integer $>Q/(Nt)$?

To locate this ``structured'' subgraph $G'$, we use an iterative ``compression'' argument, roughly inspired by the papers of Erd\H os-Ko-Rado \cite{erdoskorado} and Dyson \cite{dyson}. With each iteration, we pass to a smaller set of vertices, where we have additional information about which primes divide them. Of course, we must ensure that we end up with a sizeable graph. We do this by judiciously choosing the new graph at each step so that it has at least as many edges as what the qualitative parameters of the old graph might naively suggest. This way the new graph will have improved ``structure'' \emph{and} ``quality''. When the algorithm terminates, we will end up with a fully structured subset of $\CS$, where we know that all large GCDs are due to a large fixed common factor. This will then allow us to exploit the condition that  $L_t(q,r)>100$ for all edges $(q,r)$. Importantly, our algorithm will also control the set $\CB_t$ in terms of the terminal edge set. Hence the savings from  the condition $L_t(q,r)>100$ in the terminal graph will be transferred to $\CB_t$, establishing \eqref{eq:DS bilinear 2}.

One technical complication is that the iterative algorithm necessitates to view $G$ as a bipartite graph. In addition, it is important to use the weights $\phi(q)/q$. We thus set
\[
\mu(\CV)=\sum_{v\in\CV} \frac{\phi(v)}{v}
\quad\text{for}\quad \CV\subset\N;
\qquad 
\mu(\CE)=\sum_{(v,w)\in\CE} \frac{\phi(v)\phi(w)}{vw}
\quad\text{for}\quad\CE\subset\N^2.
\] 

\smallskip

Let us now explain the algorithm in more detail. We set $\CV_0=\CW_0=\CS$ and construct two decreasing sequences of sets $\CV_0\supseteq\CV_1\supseteq\CV_2\supseteq\cdots$ and $\CW_0\supseteq\CW_1\supseteq\CW_2\supseteq\cdots$, as well as a sequence of distinct primes $p_1,p_2,\dots$ such that either $p_j$ divides all elements of $\CV_j$, or it is coprime to all elements of $\CV_j$ (and similarly with $\CW_j$). Since $\CS$ consists solely of square-free integers, there are integers $a_j,b_j$ dividing $p_1\cdots p_j$, and such that $\gcd(v,p_1\cdots p_j)=a_j$ and $\gcd(w,p_1\cdots p_j)=b_j$ for all $v\in\CV_j$ and all $w\in\CW_j$.

Assume we have constructed $\CV_i,\CW_i,p_i$ as above for $i=1,\dots,j$. Let $\CE_i=\CB_t\cap(\CV_i\times\CW_i)$ be the edge sets. We then pick a new prime $p_{j+1}$ that occurs as common factor of $\gcd(v,w)$ for at least one edge $(v,w)\in \CE_j$. (If there is no such prime, the algorithm terminates.) Then, we pick $\CV_{j+1}$ to be either $\CV_j^{(1)}:=\{v\in \CV_j:p_{j+1}|v\}$ or $\CV_j^{(0)}:=\{v\in \CV_j:p_{j+1}\nmid v\}$ (and similarly with $\CW_{j+1}$). Deciding how to make this choice is the most crucial part of the proof and we will analyze it in more detail below. At any rate, it is clear that after a finite number of steps, this process will terminate. We will thus arrive at sets of vertices $\CV_J$ and $\CW_J$ where $a=a_J$ divides all elements of $\CV_J$, $b=b_J$ divides all elements of $\CW_J$, and $\gcd(v,w)=\gcd(a,b)$ for all edges $(v,w)\in\CE_J\subseteq \CB_t$. In particular, $\gcd(a,b)>Q/(Nt)$, as long as $\CE_J\neq\emptyset$. We have thus found our fixed large common divisor, so that we can use the Erd\H os--Vaaler argument as in \S\ref{sec:EV} to control the size of $\CE_J$. If we can ensure that $\CE_J$ is a large enough portion of $\CE_0=\CB_t$, we will have completed the proof.

Let us now explain how to make the choice of which subgraph to focus on each time. Let $G_j=(\CV_j,\CW_j,\CE_j)$ be the bipartite graph at the $j$-iteration. Because we will use an unbounded number of iterations, it is important to ensure that $G_{j+1}$ has more edges than ``what the qualitative parameters of $G_j$ would typically predict''. One way to assign meaning to this vague phrase is to use the edge density $\#\CE_j/(\#\CV_j \#\CW_j )$. Actually, in our case, we should use the weighted density 
\[
\delta_j =\frac{\mu(\CE_j)}{\mu(\CV_j)\mu(\CW_j)} .
\]
Naively, we might guess that $\mu(\CE_{j+1})\approx \delta_j \mu(\CV_{j+1})\mu(\CW_{j+1})$, meaning that $\delta_{j+1}\approx \delta_j$. So we might try to choose $G_{j+1}$ so that $\delta_{j+1}\ge \delta_j$. This would be analogous to Roth's ``density increment'' strategy \cite{roth1,roth2}. Unfortunately, such an argument loses all control over the size of $\CE_j$, so we cannot use information on $\CE_J$ to control $\CE_0=\CB_t$ (which is our end goal).

In a completely different direction, we can use the special GCD structure of our graphs to come up with another ``measure of quality'' of our new graph compared to the old one. Recall the integers $a_{j+1}$ and $b_{j+1}$. We then have
\eq{\label{eq:Ej-ub}
\#\CE_j
	\le \#\Big\{ m\le \frac{2Q}{a_j},\  n\le \frac{2Q}{b_j } :  \gcd(m,n)> \frac{Q/(Nt)}{ \gcd(a_j,b_j)  } \Big\}.
}
If all pairs $(m,n)$ on the right-hand side of \eqref{eq:Ej-ub} were due to a fixed divisor of size $>[Q/(Nt)]/\gcd(a_j,b_j)$, then we would conclude that
\[
\#\CE_j \ll t^2N^2 \cdot \frac{\gcd(a_j,b_j)^2}{a_jb_j}.
\]
Actually, Green and Walker \cite{GW} proved recently that this bound is true, even without the presence of a universal divisor. So it makes sense to consider the quantity $\#\CE_j \cdot a_jb_j/\gcd(a_j,b_j)^2$ as a qualitative measure of $G_j$. As a matter of fact, since we are weighing $v$ with $\phi(v)/v$, and we have $\phi(v)/v\le \phi(a_j)/a_j$ whenever $a_j|v$, we may even consider
\[
\lambda_j:= \frac{a_jb_j}{\gcd(a_j,b_j)^2} \cdot \frac{a_jb_j}{\phi(a_j)\phi(b_j)}  \cdot \mu(\CE_j) .
\]

Let us see a different argument for why this quantity might be a good choice, by studying the effect of each prime $p\in\{p_1,\dots,p_j\}$ to the parameters $Q/a_j$, $Q/b_j$ and $[Q/(Nt)]/\gcd(a_j,b_j)$ that control the size of $m$,  $n$, and $\gcd(m,n)$, respectively, in \eqref{eq:Ej-ub}:

\begin{itemize}
	\setlength\itemsep{0.5em}
	\item 
	\textit{Case 1: $p|a_j$ and $p|b_j$.} Then $p$ reduces the upper bounds on the size of both $m$ and $n$ by a factor $1/p$. On the other hand, it also reduces the lower bound on their GCD (that affects both $m$ and $n$) by $1/p$. Hence, we are in a balanced situation.
	
	\item 
	\textit{Case 2: $p\nmid a_j$ and $p\nmid b_j$.} In this case, $p$ affects no parameters.
	
	\item 
	\textit{Case 3: $p| a_j$ and $p\nmid b_j$.} Then $p$ reduces the upper bound on $m$ by a factor $1/p$, but it does not affect the bound on $n$ nor on $\gcd(m,n)$. This is an advantageous situation, gaining us a factor of $p$ compared to what we had. Accordingly, $\lambda_j$ is multiplied by $p$ in this case. This gain allows us to afford a big loss of vertices when falling in this ``asymmetric'' case (a proportion of $1-O(1/p)$).
	
	\item 
	\textit{Case 4: $p\nmid a_j$ and $p|b_j$.} Then we gain a factor of $p$ as in the previous case.
\end{itemize}

Iteratively increasing $\lambda_j$ would be adequate for showing \eqref{eq:DS bilinear 2}, by mimicking the Erd\H os--Vaaler argument from \S\ref{sec:EV}. Unfortunately it is not possible to guarantee that $\lambda_j$ increases at each stage because it is not sensitive enough to the edge density, and so this proposal also fails. However, we will show that (a small variation of ) the hybrid quantity
\eq{\label{eq:quality}
q_j:=\delta_j^9 \lambda_j ,
}
can be made to increase at each step, while keeping track of the sizes of the vertex sets. We call $q_j$ the \emph{quality} of $G_j$. 

Let us now discuss how we might carry out the ``quality increment'' strategy. Given $\CV_j$ and $p_{j+1}$, we wish to set $\CV_{j+1}=\CV_j^{(k)}$ and $\CW_{j+1}=\CW_j^{(\ell)}$ for some $k,\ell\in\{0,1\}$. Let us call $G_j^{(k,\ell)}$ each of the four potential choices for $G_{j+1}$. For their quality $q_j^{(k,\ell)}$, we have:
\[
\begin{array}{ll}
\ds\frac{q_j^{(1,1)}}{q_j} = \Big(\frac{\delta^{(1,1)}}{\delta_j}\Big)^{10} \alpha \beta \Big(1-\frac{1}{p}\Big)^{-2} ,

&\ds \quad \frac{q_j^{(1,0)}}{q_j} = \Big(\frac{\delta_j^{(1,0)}}{\delta_j}\Big)^{10}\alpha (1-\beta) p \Big(1-\frac{1}{p}\Big)^{-1},

\smallskip 

\\

\ds \frac{q_j^{(0,1)}}{q_j} = \Big(\frac{\delta_j^{(0,1)}}{\delta_j}\Big)^{10} (1-\alpha) \beta p \Big(1-\frac{1}{p}\Big)^{-1}  ,
&\ds\quad  \frac{q_j^{(0,0)}}{q_j} = \Big(\frac{\delta_j^{(0,0)}}{\delta_j}\Big)^{10} (1-\alpha)(1-\beta)  ,
\end{array}
\]
where $\delta_j^{(k,\ell)}$ is the edge density of $G^{(k,\ell)}_j$, $\alpha=\mu(\CV_j^{(1)})/\mu(\CV_j)$ is the proportion of vertices in $\CV_j$ that are divisible by $p_{j+1}$, and similarly $\beta=\mu(\CW_j^{(1)})/\mu(\CW_j)$. In addition, we have
\[
\delta_j^{(1,1)}\alpha\beta+\delta_j^{(1,0)}\alpha(1-\beta)+\delta_j^{(0,1)} (1-\alpha)\beta+\delta_j^{(0,0)}(1-\alpha)(1-\beta) = \delta_j,
\]
so that if one of the $\delta_j^{(k,\ell)}$'s is smaller than $\delta$, some other must be larger. Such an unbalanced situation is advantageous, so let us assume that $\delta_j^{(k,\ell)}\sim \delta_j$ for all $k,\ell$. 

Notice that we have an extra factor $p$ in the asymmetric cases $(0,1)$ and $(1,0)$. We can then easily obtain a quality increment in one of these two cases, unless $\alpha,\beta\ll 1/p$, or if $\alpha,\beta\ge 1-O(1/p)$. It turns out that the former case can be treated with a trick. So, the real difficulty is to obtain a quality increment when $\alpha,\beta$ are both close to 1. As a matter of fact, the critical case is when $\alpha,\beta\sim 1-1/p$. Indeed, we then have $q_j^{(k,\ell)}\sim1$ in all four cases, so we maintain a constant quality no matter which subgraph we choose to focus on. 

It is important to remark here that the factor $(1-1/p)^{-2}$ in the case $(k,\ell)=(1,1)$ is essential (the factors $(1-1/p)^{-1}$ in the asymmetric cases are less important as it turns out). Without this extra factor, we would not have been able to guarantee that the quality stays at least as big as $q_j$. Crucially, this factor originates from the weights $\phi(v)/v$ of the vertices that are naturally built in the Duffin--Schaeffer conjecture and that dampen down contributions from integers with too many prime divisors.

We conclude this discussion by going back to the Model Problem. We mentioned in \S\ref{sec:EV} that this problem is false. The reason is a counterexample due to Sam Chow, a square-free version of which is given by $\CS=\{P/j : j|P,\ x/2\le j\le x\}$ with $P=\prod_{p\le x}p$. Indeed, all pairwise GCDs here are $\ge P/x^2$, but there is no fixed integer of size $\gg P/x^2$ dividing a positive proportion of elements of this set. In addition, note that if $p\le x/\log x$, then the proportion of $\CS$ divisible by $p$ is $\sim1-1/p$, just like in the critical case discussed above.

\subsection{The quality increment argument}

We now discuss the formal details of our iterative algorithm. We must first set up some notation. 
We say that $G=(\CV,\CW,\CE,\CP,a,b)$ is a \emph{square-free GCD graph} if:
\begin{itemize}[itemsep=0pt]
	\item $\CV$ and $\CW$ are non-empty, finite sets of square-free integers;
	\item $(\CV,\CW,\CE)$ is a bipartite graph, meaning that $\CE\subseteq\CV\times\CW$;
	\item $\CP$ is a finite set of primes, and $a$ and $b$ divide $\prod_{p\in\CP}p$;
	\item $a|v$ and $b|w$ for all $(v,w)\in\CV\times\CW$;
	\item if $(v,w)\in\CE$ and $p\in\CP$, then $p|\gcd(v,w)$ precisely when $p|\gcd(a,b)$.
\end{itemize}
We shall refer to $(\CP,a,b)$ as the \emph{multiplicative data} of $G$. Furthermore, we defined the \emph{edge density} of $G$ by $\delta(G):=\frac{\mu(\CE)}{\mu(\CV)\mu(\CW)}$, and its \emph{quality} by
\[
q(G):= \delta(G)^9\cdot  \mu(\CE)\cdot  \frac{ab}{\gcd(a,b)^2} \cdot \frac{ab}{\phi(ab)} \cdot \prod_{p\in\CP} \Big(1-\frac{1}{p^{3/2}}\Big)^{-10} .
\]
In addition, we define the set of ``remaining large primes'' of $G$ by 
\[
\CR(G):=\{p\notin\CP: p>5^{100},\ p|\gcd(v,w)\ \text{for some}\ (v,w)\in\CE\} .
\]
Finally, if $G'=(\CV',\CW',\CE',\CP',a',b')$ is another square-free GCD graph, we call it a \emph{subgraph} of $G$ if $\CV'\subseteq\CV$, $\CW'\subseteq\CW$, $\CE'\subseteq \CE$, $\CP'\supseteq\CP$, $\prod_{p|a',\,p\in\CP}p=a$, $\prod_{p|b',\,p\in\CP}p=b$.

\begin{lem}[The quality increment argument]\label{lem:quality-increment} 
	Let $G=(\CV,\CW,\CE,\CP,a,b)$  be a square-free GCD graph, let $p\in \CR(G)$, and let $\alpha=\frac{\mu(\{v\in\CV:p|v\})}{\mu(\CV)}$ and $\beta=\frac{\mu(\{w\in\CW:p|w\})}{\mu(\CW)}$.
\begin{enumerate}
	\item If $\min\{\alpha,\beta\}\le 1-5^{12}/p$, then there is a subgraph $G'$ of $G$ with multiplicative data $(\CP\cup\{p\},ap^k,bp^\ell)$ for some $k,\ell\in\{0,1\}$ satisfying $\delta(G')^mq(G') \ge 2^{k\neq\ell} \delta(G)^mq(G)$ for $m\in\{0,1\}$.
	\item If $\min\{\alpha,\beta\}>1-5^{12}/p$, then there is a subgraph $G'$ of $G$ with set of primes $\CP\cup\{p\}$ and with quality 
	$q(G') \ge q(G)$.
\end{enumerate}	
\end{lem}

\begin{proof} Each $k,\ell\in\{0,1\}$ define a subgraph of $G$ with multiplicative data $(\CP\cup\{p\},ap^k,bp^\ell)$. Indeed, we merely need to focus on the vertex subsets $\CV_k=\{v\in \CV: p^k\|v\}$ and $\CW_\ell=\{w\in\CW:p^\ell\|w\}$. This new GCD graph is formally given by the sextuple $G_{k,\ell}=(\CV_k,\CW_\ell,\CE_{k,\ell},\CP\cup\{p\},ap^k,bp^\ell)$, where $\CE_{k,\ell}=\CE\cap(\CV_k\times\CW_\ell)$. Let $\delta_{k,\ell}=\frac{\mu(\CE_{k,\ell})}{\mu(\CE)}$, $\alpha_k=\frac{\mu(\CV_k)}{\mu(\CV)}$ and $\beta_\ell=\frac{\mu(\CW_\ell)}{\mu(\CW)}$, so that $\alpha_0=1-\alpha$, $\alpha_1=\alpha$, $\beta_0=1-\beta$ and $\beta_1=\beta$. We then have
\eq{\label{eq:quality-ratio}
	\frac{\delta(G_{k,\ell})^mq(G_{k,\ell})}{\delta(G)^mq(G)} =  
	\frac{\delta_{k,\ell}^{10+m} (\alpha_k\beta_\ell)^{-9-m}p^{1_{k\neq\ell}}}{ (1-1/p)^{k+\ell}(1-p^{-3/2})^{10}},
}

\smallskip

(a) We claim that there exist choices of $k,\ell\in\{0,1\}$ such that
\eq{\label{eq:quality-edge density}
\delta_{k,\ell} \ge 
		\begin{cases} 
			(\alpha_k\beta_k)^{9/10} &\text{if}\ k=\ell,\\
			\ds \frac{\alpha(1-\beta)+(1-\alpha)\beta}{5}	&\text{if}\ k\neq\ell.
		\end{cases}
}
To prove this claim, it suffices to show that
\eq{\label{eq:inequality}
(\alpha\beta)^{9/10}+((1-\alpha)(1-\beta))^{9/10} +\frac{2}{5}\cdot [\alpha(1-\beta)+(1-\alpha)\beta]	 \le 1.
}
Let $u=\max\{\alpha\beta,(1-\alpha)(1-\beta)\}$. Then
\[
(\alpha\beta)^{9/10}+((1-\alpha)(1-\beta))^{9/10} \le u^{2/5} [(\alpha\beta)^{1/2}+((1-\alpha)(1-\beta))^{1/2}]\le u^{2/5} 
\]
by the Cauchy--Schwarz inequality. On the other hand, we have
\[
\alpha(1-\beta)+(1-\alpha)\beta = 1-\alpha\beta-(1-\alpha)(1-\beta)\le 1-u.
\]
In conclusion, the left-hand side of \eqref{eq:inequality} is $\le u^{2/5}+2(1-u)/5\le 1$, as needed.

Now, if \eqref{eq:quality-edge density} is true with $k=\ell$, part (a) of the lemma follows immediately by \eqref{eq:quality-ratio} upon taking $G'=G_{k,k}$. Assume then that \eqref{eq:quality-edge density} fails when $k=\ell$. We separate  two cases.

\smallskip 

\noindent\emph{Case 1: $\max\{\alpha,\beta\}>5^{12}/p$.} We know that \eqref{eq:quality-edge density} holds for some choice of $k\neq\ell$. Suppose that $k=1$ and $\ell=0$ for the sake of concreteness; the other case is similar. Then, \eqref{eq:quality-ratio} implies
\[
\frac{\delta(G_{1,0})^mq(G_{1,0})}{\delta(G)^mq(G)} \ge \Big(\frac{\alpha(1-\beta)+\beta(1-\alpha)}{5}\Big)^{10+m} \frac{p}{(\alpha(1-\beta))^{9+m}}
\ge \frac{\alpha(1-\beta)+\beta(1-\alpha)}{5^{11}} p
\]
for $m\le1$. The proof is complete by taking $G'=G_{1,0}$, unless $\alpha(1-\beta)+\beta(1-\alpha)<2\cdot 5^{11}/p$. In this case, we claim that either $\max\{\alpha,\beta\}<5^{12}/p$ or $\min\{\alpha,\beta\}>1-5^{12}/p$ (both of which we have assumed are false). By symmetry, we may assume $\alpha\le 1/2$. Then $\beta/2\le \beta(1-\alpha)< 2\cdot 5^{11}/p$, as needed. In particular, $\beta\le 1/2$ (because $p>5^{100}$), and thus $\alpha/2\le \alpha(1-\beta)<2\cdot 5^{11}/p$. We have thus reached a contradiction. This proves the lemma in this case.

\smallskip

\noindent\emph{Case 2: $\max\{\alpha,\beta\}\le 5^{12}/p$.} We must then have $\delta_{1,1}\le (\alpha\beta)^{9/10}\le 5^{22} p^{-9/5}$. Let us now define the GCD subgraph $G'=(\CV,\CW,\CE',\CP\cup\{p\},a,b)$, where $\CE'=\CE\setminus(\CV_1\times\CW_1)$. Notice that we trivially have $a|v$ and $b|w$. In addition, since we have removed all edges $(v,w)$ where $p$ divides both $v$ and $w$, we must have that $p\nmid \gcd(v,w)$ whenever $(v,w)\in\CE'$. So, indeed, we see that $G'$ is a GCD subgraph of $G$. Moreover, 
\[
\frac{\delta(G')^mq(G')}{\delta(G)^mq(G)} = \Big(\frac{\mu(\CE')}{\mu(\CE)}\Big)^{10+m}  (1-p^{-3/2})^{-10}
= (1-\delta_{1,1})^{10+m} (1-p^{-3/2})^{-10}\ge 1
\]
for $m\le1$, because $\delta_{1,1}\le 5^{22}p^{-9/5}$ and $p>5^{100}$. This proves the lemma in this case too.

\smallskip

(b) Let $c=(1-p^{-3/2})^{-1}$. Using \eqref{eq:quality-ratio}, we get a quality increment by letting $G'=G_{k,\ell}$ if one of the following inequalities holds:
\eq{\label{eq:quality-increment case 2}
\begin{array}{ll}
c\delta_{1,1}\ge (\alpha\beta)^{9/10}(1-1/p)^{2/10}; \quad&
c\delta_{0,0}\ge ((1-\alpha)(1-\beta))^{9/10};\\
c\delta_{1,0}\ge (\alpha(1-\beta))^{9/10}p^{-1/10};\quad &
c\delta_{0,1}\ge ((1-\alpha)\beta)^{9/10}p^{-1/10}.
\end{array}
}
Let $\alpha=1-A/p$ and $\beta=1-B/p$ with $A,B\in[0,5^{12}]$. It suffices to show that
\eq{\label{eq:quality-increment e2}
c\ge\Big(1-\frac{A}{p}\Big)^{\frac{9}{10}}\Big(1-\frac{B}{p}\Big)^{\frac{9}{10}}\Big(1-\frac{1}{p}\Big)^{\frac{2}{10}} + \frac{(AB)^{\frac{9}{10}}}{p^{9/5}} + \frac{(1-\frac{A}{p})^{\frac{9}{10}}B^{\frac{9}{10}}+A^{\frac{9}{10}}(1-\frac{B}{p})^{\frac{9}{10}}}{p} .
}
Indeed, the right-hand side of  \eqref{eq:quality-increment e2}  is
\als{
&\le \exp\Big(-\frac{0.9 A+0.9B+0.2}{p}\Big) + \frac{5^{22}}{p^{9/5}} + \frac{A^{9/10}+B^{9/10}}{p} \\
&\le 1-\frac{0.A+0.9B+0.2}{p}+\frac{(0.9A+0.9B+0.2)^2}{2p^2}+ \frac{5^{22}}{p^{9/5}} + \frac{A^{9/10}+B^{9/10}}{p} ,
}
where we used the inequalities $0\le 1-x\le e^{-x}\le 1-x+x^2/2$, valid for all $x\in[0,1]$.
By the inequality of arithmetic and geometric means, we have $0.9A+0.1\ge A^{9/10}$ and $0.9B+0.1\ge B^{9/10}$. Hence, the right-hand side of  \eqref{eq:quality-increment e2}  is
\[
\le 1+\frac{(0.9A+0.9B+0.2)^2}{2p^2}+ \frac{5^{11}}{p^{9/5}} \le 1+\frac{5^{25}}{p^2}+ \frac{5^{22}}{p^{9/5}} \le 1+ \frac{1}{p^{3/2}}\le c
\]
for $p>5^{100}$. This completes the proof of the part (b) of the lemma.
\end{proof}


\subsection{Proof of Theorem \ref{thm:DS-special-case}}
	
Let $Q$, $N$, $\CS$ and $\CB_t$ with $t\ge t_{j_0+1}$ be as in \S\ref{sec:EV}. We want to prove \eqref{eq:DS bilinear 2}. We may assume that $\mu(\CB_t)\ge \mu(\CS)^2/t$; otherwise, \eqref{eq:DS bilinear 2} is trivially true.

Consider the GCD graph $G_0:=(\CS,\CS,\CB_t,\emptyset, \emptyset, \emptyset)$, and note that $\delta(G_0)\ge1/t$. We repeatedly apply part (a) of Lemma \ref{lem:quality-increment} to create a sequence of distinct primes $p_1,p_2,\dots$ and of square-free GCD graphs $G_j=(\CV_j,\CW_j,\CE_j,\{p_1,\dots,p_j\},a_j,b_j)$, $j=1,2,\dots$, with $G_j$ a subgraph of $G_{j-1}$. Assuming we have applied Lemma \ref{lem:quality-increment} (a) $j$ times, we may apply it once more if there is $p\in\CR(G_j)$ dividing a proportion $\le 1-5^{12}/p$ of $\CV_j$ and $\CW_j$.

Naturally, the above process will terminate after a finite time, say after $J_1$ steps and we will arrive at a GCD graph $G_{J_1}$ such that if $p\in\CR(G_{J_1})$, then $p$ divides a proportion $>1-5^{12}/p$ of the vertex sets $\CV_{J_1}$ and $\CW_{J_1}$.  In addition, the sequence of GCD graphs produced is such that $\delta(G_j)^mq(G_j)\ge 2^{1_{j\in\CD}} \delta(G_{j-1})^mq(G^{(j-1)})$ for $m\in\{0,1\}$, where 
\[
\CD=\{j\le J_1: p_j\ \text{divides}\ a_{J_1}b_{J_1}/\gcd(a_{J_1},b_{J_1})^2\} .
\]
In particular,
\eq{\label{eq:total-quality-increment after stage 1}
\delta(G_{J_1})^mq(G_{J_1}) \ge 2^{\#\CD} \delta(G_0)^mq(G_0) \quad\text{for}\ m\in\{0,1\}.
}
To proceed, we must separate two cases.

\smallskip

\noindent\textit{Case 1: $q(G_{J_1})\ge t^{30} q(G_0)$.} We apply repeatedly Lemma \ref{lem:quality-increment} (either part (a) or (b), according to whether the condition $\min\{\alpha,\beta\}\le 1-5^{12}/p$ holds or fails) to create a sequence of primes $p_{J_1+1},p_{J_1+2},\dots$ that are distinct from each other and from $p_1,\dots,p_{J_1}$, and of square-free GCD graphs $G_j=(\CV_j,\CW_j,\CE_j,\{p_1,\dots,p_j\},a_j,b_j)$, $j=J_1+j,J_1+2,\dots$, with $G_j$ a subgraph of $G_{j-1}$.
As before, this process will terminate, say after $J_2-J_1$ steps, and we will arrive at a GCD graph $G_{J_2}$ with  $\CR(G_{J_2})=\emptyset$. By construction, we have
\[
q(G_{J_2}) \ge q(G_{J_2-1}) \ge \cdots \ge q(G_{J_1})\ge t^{30} q(G_0). 
\]
In addition, we have 
\eq{\label{eq:quality at starting graph}
q(G_0) = \frac{\mu(\CB_t)^{10}}{\mu(\CS)^{18}} \ge \frac{\mu(\CB_t)^{10}}{N^{18}}
}
by \eqref{eq:DS-total weight}. (In particular, note that $q(G_0)>0$, so $q(G_{J_2})>0$ and thus $\CE_{J_2}\neq\emptyset$.) On the other hand, if we let $a=a_{J_2}$ and $b=b_{J_2}$, then $\gcd(v,w)|\gcd(a,b)P$ with $P=\prod_{p\le 5^{100}}p$ for all $(v,w)\in\CE_{J_2}$. In particular, $\gcd(a,b)>Q/(PNt)$. Moreover,
\[
\mu(\CE_{J_2}) \le \mathop{\sum\sum}_{m\le 2Q/a,\ n\le 2Q/b}\frac{\phi(am)}{am} \cdot \frac{\phi(bn)}{bn} \le \frac{\phi(a)\phi(b)}{ab}\cdot \frac{4Q^2}{ab} .
\]
Since $\delta(G_{J_2})\le1$ and $\prod_p(1-1/p^{3/2})^{-10}<\infty$, we then have
\eq{	\label{eq:quality at end graph}
q(G_{J_2}) 
	\ll  \mu(\CE_{J_2}) \frac{ab}{\gcd(a,b)^2} \frac{ab}{\phi(a)\phi(b)} \ll t^2N^2 .
}
Recalling that $q(G_{J_2})\ge t^{30}q(G_0)$, relations \eqref{eq:quality at starting graph} and \eqref{eq:quality at end graph} complete the proof of \eqref{eq:DS bilinear 2}, and thus of Theorem \ref{thm:DS-special-case} in this case.

\smallskip

\noindent\textit{Case 2: $q(G_{J_1})<t^{30}q(G_0)$.} In this case, we do not have such a big quality gain, so we need to use that $L_t(v,w)>100$ for all $(v,w)\in\CB_t$. But we must be very careful because this condition might be dominated by the prime divisors of the fixed integers $a$ and $b$ we are constructing. Before we proceed, note that \eqref{eq:total-quality-increment after stage 1} implies that
\eq{\label{eq:density lower bound}
\delta(G_{J_1}) \ge \delta(G_0) \cdot \frac{q(G_0)}{q(G_{J_1})} \ge \frac{1}{t} \cdot \frac{1}{t^{30}} =  \frac{1}{t^{31}} .
}

Let $\CR=\CR(J_1)$ and let $p\in\CR$. By the construction of $G_{J_1}$, $p$ divides a proportion $>1-5^{12}/p$ of the vertex sets $\CV_{J_1}$ and $\CW_{J_1}$. Therefore, 
\[
\mu\Big(\big\{(v,w)\in\CE_{J_1}   : p|vw/\gcd(v,w)^2 \big\}\Big) 
	\le \frac{2\cdot 5^{12}}{p} \mu(\CV_{J_1})\mu(\CW_{J_1}) \le  \frac{5^{13}t^{31}}{p} \mu(\CE_{J_1}),
\]
where we used \eqref{eq:density lower bound}. As a consequence, we find
\[
\sum_{(v,w)\in \CE_{J_1}} \frac{\phi(v)\phi(w)}{vw} \sum_{\substack{p>t^{32},\ p\in\CR\\ p|vw/\gcd(v,w)^2}} \frac{1}{p} 
	\le \sum_{p>t^{32}} \frac{5^{13}t^{31} \mu(\CE_{J_1})}{p^2} 	\le \frac{\mu(\CE_{J_1})}{100}.
\]
Hence, if we let 
\[
\CE_{J_1}^{\text{good}}=\Big\{(v,w)\in \CE_{J_1} : \sum_{\substack{p>t^{32},\ p\in\CR \\ p|vw/\gcd(v,w)^2}} \frac{1}{p} \le 1 \Big\},
\]
Markov's inequality implies that $\mu(\CE_{J_1}^{\text{good}})\ge 0.99 \mu(\CE_{J_1})$. We then define the GCD graph $G_0'=(\CV_{J_1},\CW_{J_1},\CE_{J_1}^{\text{good}},\CP,a_{J_1},b_{J_1})$. Note that
\[
g(G_0')=\bigg(\frac{\mu(\CE_{J_1}^{\text{good}})}{\mu(\CE_{J_1})} \bigg)^{10} q(G_{J_1}) \ge \frac{q(G_{J_1})}{2} \ge \frac{q(G_0)}{2} .
\]

Next, we apply repeatedly Lemma \ref{lem:quality-increment} to create a sequence of distinct primes $p_{J_1+1},p_{J_1+2},\dots\in\CR$ and of GCD graphs $G_j'=(\CV'_j,\CW'_j,\CE'_j,\{p_1,\dots,p_{J_1+j}\},a'_j,b'_j)$, $j=1,\dots$, with $G_j'$ a subgraph of $G_{j-1}'$. This process will terminate, say after $K$ steps, and we will arrive at a GCD graph $G'_K$ with $\CR(G'_K)=\emptyset$. By construction, we have
\eq{\label{eq:quality-increment after stage 2(b)}
		q(G'_K) \ge q(G'_{K-1}) \ge \cdots \ge q(G'_0)\ge q(G_0)/2>0.
}
In particular, $\CE_K'\neq\emptyset$. It remains to give an upper bound on $q(G'_K)$. 

Let $a'=a_K'$  and $b=b_K'$, and recall that $P=\prod_{p\le 5^{100}}p$. Then  $\gcd(v,w)|\gcd(a',b')P$ for all $(v,w)\in\CE_K'$. In particular, $\gcd(a',b')>Q/(PNt)$. Moreover, if $(v,w)\in\CE_K'$ and we let $v=a'm$ and $w=b'n$, then
\[
100<L_t(v,w) \le 5+L_{t^{32}}(v,w) \le 6 + \sum_{\substack{p>t^{32},\ p\notin\CR \\ p|vw/\gcd(v,w)^2 }}\frac{1}{p}
		\le 6+ \frac{\#\CD}{t^{32}} + L_{t^{32}}(m,n) ,
\]
where the first inequality is true because $(v,w)\in \CB_t$, the second one because $\sum_{y<p\le y^2}1/p\le1$ for $y\ge t_{j_0}$, the third one because $(v,w)\in \CE_{J_1}^{\text{good} }$, and the fourth one because if $p$ divides $a'b'/\gcd(a',b')^2$ and $p\notin\CR$, then $p\in\CD$. Now, since $2^{\#\CD}\le q(G_{J_1})/q(G_0)\le t^{30}$, we have that $L_{t^{32}}(m,n)>93$. Therefore,
	\[
	\mu(\CE_K') \le \mathop{\sum\sum}_{\substack{m\le 2Q/a',\ n\le 2Q/b' \\ L_{t^{32}}(m,n)>93}} 
	\frac{\phi(a')\phi(b')}{a'b'} \ll \frac{\phi(a')\phi(b')}{a'b'}\cdot \frac{4Q^2}{a'b'} e^{-t^{32}},
	\]
	by arguing as in the proof of \eqref{eq:DS for full density}. We may then insert this inequality into the definition of $q(G_K')$ and conclude that $q(G_K')\ll e^{-t^{32}} t^2N^2 $. Together with \eqref{eq:quality-increment after stage 2(b)} and \eqref{eq:quality at starting graph}, this completes the proof of \eqref{eq:DS bilinear 2}, and thus of Theorem \ref{thm:DS-special-case} in this last case as well.

\section*{Acknowledgments}
The author is grateful to James Maynard for his comments on a preliminary version of the paper, and also to Andrew Granville for pointing out some inaccuracies in the published version.

\section*{Funding}

The author is supported by the \emph{Courtois Chair II in fundamental research} of the Universit\'e de Montr\'eal, the \emph{Natural Sciences and Engineering Research Council of Canada} (Discovery Grant  2018-05699), and the \emph{Fonds de recherche du Qu\'ebec - Nature et technologies} (projets de recherche en \'equipe 256442 and 300951).


\end{document}